\documentclass[reqno]{amsart} %reqno右侧编号，公式居中
\usepackage{mathrsfs} %数学字体
\usepackage{hyperref} %引用超链接

\usepackage{cases}

\makeatletter
\@namedef{subjclassname@2020}{%
	\textup{2020} Mathematics Subject Classification}
\makeatother

\makeatletter

\newcommand{\Rmnum}[1]{\expandafter\@slowromancap\romannumeral #1@}
\makeatother

\def\Xint#1{\mathchoice
	{\XXint\displaystyle\textstyle{#1}}%
	{\XXint\textstyle\scriptstyle{#1}}%
	{\XXint\scriptstyle\scriptscriptstyle{#1}}%
	{\XXint\scriptscriptstyle\scriptscriptstyle{#1}}%
	\!\int}
\def\XXint#1#2#3{{\setbox0=\hbox{$#1{#2#3}{\int}$}
		\vcenter{\hbox{$#2#3$}}\kern-.5\wd0}}

\def\dashint{\Xint-}

\begin{document}
\title[\hfil  Asymptotic mean value properties\dots] {Asymptotic mean value properties for the elliptic and parabolic double phase equations}

\author[W. Meng and C. Zhang \hfil \hfilneg]
{Weili Meng and Chao Zhang$^*$}
	
\thanks{$^*$Corresponding author.}
	
\address{Weili Meng \hfill\break
School of Mathematics, Harbin Institute of Technology, Harbin 150001, P.R. China} \email{1190500121@stu.hit.edu.cn}
	
\address{Chao Zhang\hfill\break
School of Mathematics, Harbin Institute of Technology,	Harbin 150001, P.R. China} \email{czhangmath@hit.edu.cn}

\date{}
	
\begin{abstract} We characterize an asymptotic mean value formula in the viscosity sense for the double phase elliptic equation
$$
-\text{\rm{div}}(\lvert \nabla u \rvert^{p-2}\nabla u+ a(x)\lvert\nabla u \rvert^{q-2}\nabla u)=0 
$$
and the normalized double phase parabolic equation
$$
u_t=\lvert\nabla u \rvert ^{2-p}\text{\rm{div}}(\lvert \nabla u \rvert^{p-2}\nabla u+ a(x,t)\lvert\nabla u \rvert^{q-2}\nabla u), \quad 1<p\leq q<\infty.
$$
This is the first mean value result for such kind of nonuniformly elliptic and parabolic equations. In addition, the results obtained can also be applied to the $p(x)$-Laplace equations and the variable coefficient $p$-Laplace type equations.
\end{abstract}

\maketitle \numberwithin{equation}{section}
\newtheorem{theorem}{Theorem}[section]
\newtheorem{lemma}[theorem]{Lemma}
\newtheorem{definition}[theorem]{Definition}
\newtheorem{claim}[theorem]{Claim}
\newtheorem{proposition}[theorem]{Proposition}
\newtheorem{remark}[theorem]{Remark}
\newtheorem{corollary}[theorem]{Corollary}
\newtheorem{example}[theorem]{Example}
\allowdisplaybreaks

\section{Introduction}\label{sec1}

Let $\Omega$ be a bounded domain in $\mathbb{R}^N (N\geq 2)$.  We consider the following double phase elliptic equation
\begin{align}\label{1-1}
-\text{\rm{div}}(\lvert \nabla u \rvert^{p-2}\nabla u+ a(x)\lvert\nabla u \rvert^{q-2}\nabla u)=0 \quad \text{in } \Omega,
\end{align}
where $1<p\leq q<\infty$ and   $a(x)\geq 0$. It is the Euler-Lagrange equation of the non-autonomous functional
$$
W^{1,1}(\Omega)\ni  w\mapsto \int_{\Omega}\left(\frac{1}{p}\lvert\nabla w\rvert^p+\frac{a(x)}{q}\lvert\nabla w\rvert^q\right)dx.
$$
Originally, this functional is connected to the Homogenization theory and Lavrentiev phenomenon \cite{JKO, Mar91, Z}, which reflects the behavior of strongly anisotropic materials, where the coefficient $a(\cdot)$ is used to regulate two mixtures with $p$ and $q$ hardening, respectively. 
 
During the last years, problems  of the type considered in \eqref{1-1}  have received great attention from the variational point of view. The regularity of minimizers and weak solutions is determined via a delicate interaction between the growth conditions and the pointwise behaviour of $a(\cdot)$.  Starting from a series of remarkable works of Colombo and Mingione  et. al. \cite{BCM18,CM,CM22015}, despite its relatively short history, double phase problems has already achieved an elaborate theory with several connections to other branches. We refer the readers to \cite{BBO20, BDMS, BO17, CS16, CM16, DeFM, DeFM21, DeFM22, DeFP19, FRZZ22, FZ210, PPR22} and the references therein.

It is well-known that a continuous function $u$ is harmonic if and only if it obeys the mean value formula discovered by Gauss. That is, $u$ solves the Laplace equation $\Delta u=0$ in $\Omega$ if and only if 
$$
u(x)=\dfrac{1}{|B_{\varepsilon}(x)|}\int_{B_{\varepsilon}(x)}u(y)dy=\dashint_{B_{\varepsilon}(x)}u(y)dy
$$
holds for all $x \in \Omega$ and $B_{\varepsilon}(x)\subset\Omega$.
In fact,  an asymptotic version of the mean value property
$$
u(x)=\dashint_{B_{\varepsilon}(x)}u(y)dy+o(\varepsilon^2) \quad \textmd{as  } \varepsilon\rightarrow0
$$
suffices to characterize harmonic functions (see \cite{Bla,Ku,Pr}). Moreover, a nonlinear mean value property was explored in  \cite{MPR1} that  continuous function $u$ is a viscosity solution of the $p$-Laplace equation
$$-\Delta_pu=-\text{\rm{div}}(\lvert \nabla u \rvert^{p-2}\nabla u)=0 \quad \text{in } \Omega$$
if and only if the asymptotic expansion 
\begin{equation*}
	u(x)=\frac{\alpha_p}{2}\left\{\mathop{\rm{max}}\limits_{\overline{B_\varepsilon(x)}}u+\mathop{\rm{min}}\limits_{\overline{B_\varepsilon(x)}}u\right\}+\beta_p\dashint_{B_\varepsilon(x)}u(y)dy+o(\varepsilon^2) \quad \text{as } \varepsilon\rightarrow 0
\end{equation*}
holds for all $x\in \Omega$ in the viscosity sense, where $\alpha_p+\beta_p=1$ and $\frac{\alpha_p}{\beta_p}=\frac{p-2}{N+2}$. The second term is a linear one, while the first term counts for the nonlinearity: the greater the $p$, the more nonlinear the formula. The expression holds in the viscosity sense, which means that when the $C^2$ test function $\phi$ with non-vanishing gradient is close to $u$ from below (above),  the expression is satisfied with $\geq$ ($\leq$) for the test function at $x$ respectively.

These mean value formulas originate from the study of dynamic programming for tug-of-war games.  The viscosity solution of the normalized parabolic $p$-Laplace equation is characterized by an asymptotic mean value formula, which is related to the tug-of-war game with noise, see \cite{FZ23,Le2,Le,MPR2,PSSW,PS}. For more related asymptotic mean value results, we refer to \cite{FLM} for $p$-harmonic functions in the Heisenberg group, \cite{BCR} for Monge-Amp\`{e}re equation, \cite{BCM,MPR2} for the nonlinear parabolic equations and the recently published monograph \cite{BR}  for historical references and more general equations.

From the results mentioned above, we can see that there are few results concerning the asymptotic mean value properties for the general nonuniformly elliptic and parabolic equations. Motivated by the previous works  \cite{MPR1,MPR2}, our intention in the present paper is to build a new bridge between the viscosity solutions and the asymptotic mean value formula for the double phase equations \eqref{1-1} and \eqref{1-6}. In addition, the method developed here can also be used to more equations, such as the $p(x)$-Laplace equations and the variable coefficient $p$-Laplace type equations. The first result is stated as follows.

\begin{theorem}\label{thm1}
Let $1<p\leq q<\infty$,  the non-negative function $a(x)$ be a $C^1$ function in $\Omega$ and let $u(x)$ be a continuous function in $\Omega$. Then Eq. \eqref{1-1}
holds in the viscosity sense if and only if the asymptotic expansion 
\begin{align}\label{1-4}
u(x)&=\frac{\alpha_p+M_u(x)\alpha_q}{2(1+M_u(x))}\left\{\max \limits_{\overline{B_\varepsilon(x)}}u+\min\limits_{\overline{B_\varepsilon(x)}}u\right\}+\dfrac{\beta_p+M_u(x)\beta_q}{1+M_u(x)}\dashint_{B_\varepsilon(x)}u(y)dy \nonumber\\
&\quad+\frac{\varepsilon\lvert\nabla u(x) \rvert^{q-p}}{4(N+p)(1+M_u(x))}\dashint_{B_\varepsilon(x)}u(y+\varepsilon\nabla a(x))-u(y-\varepsilon\nabla a(x))dy+o(\varepsilon^2)
\end{align}
as $\varepsilon\rightarrow 0$, holds for all $x\in \Omega$ in the viscosity sense. Here 
\begin{align}\label{1-5}
&\alpha_p+\beta_p=1, \quad \frac{\alpha_p}{\beta_p}=\frac{p-2}{N+2}, \nonumber\\
&\alpha_q+\beta_q=1, \quad \frac{\alpha_q}{\beta_q}=\frac{q-2}{N+2},  \\
&M_u(x)=a(x)\dfrac{N+q}{N+p}\lvert\nabla u(x) \rvert^{q-p} \nonumber.
\end{align}
\end{theorem}

\begin{remark}
Note that here we only require that $u$ is a continuous function. However,  $\nabla u$ appears in the formula \eqref{1-4} due to  the fact that it is an expression in the viscosity sense. In other words, we focus on the $C^2$ test function $\phi$ that approaches $u$ from above and below. More details will be given in Section \ref{sec2}.
\end{remark}
 
\begin{remark}
From the formula \eqref{1-4}, we can see that all the terms on the right-hand side are nonlinear, which is different from the standard $p$-Laplace equation.  The exponents $p, q$ and the non-negative coefficient  $a(x)$ coupling together influence on the nonlinearity in a delicate way. In particular, when $a(x)$ is a positive constant, Eq. \eqref{1-1} is nothing but the $(p,q)$-Laplace equation. Then the third term 
$$
\frac{\varepsilon\lvert\nabla u(x) \rvert^{q-p}}{4(N+p)(1+M_u(x))}\dashint_{B_\varepsilon(x)}u(y+\varepsilon\nabla a(x))-u(y-\varepsilon\nabla a(x))dy
$$
will vanish.
\end{remark}

Next, we turn to the parabolic case. Let $T>0$, $\Omega_T=\Omega\times(0,T)$ be a space-time cylinder, and let $a(x,t)\geq 0$ be a function that is $C^1$ in the space variable and continuous in the time variable, respectively. We consider the following parabolic equation
\begin{equation}\label{1-6}
u_t=\lvert\nabla u \rvert ^{2-p}\text{\rm{div}}(\lvert \nabla u \rvert^{p-2}\nabla u+ a(x,t)\lvert\nabla u \rvert^{q-2}\nabla u)\quad \text{\rm{in }}\Omega_T,
\end{equation}
which is called the \textit{normalized double phase} parabolic equation. The difference between elliptic and  the normalized parabolic case is that we have to consider the influence of time variable $t$ in parabolic setting. To this end, we try to separate the estimates according to $p$ and $q$, and consider the integrals in  different time intervals. Finally, we find that when the two time lags satisfy certain viscosity condition, $u$ satisfies the asymptotic mean value formula in the viscosity sense is equivalent to $u$ is the viscosity solution to Eq.  \eqref{1-6}. The second result is stated as follows.

\begin{theorem}\label{thm2}
Let $1<p\leq q<\infty$ , the non-negative function $a(x,t)$ be a function that is $C^1$ in the space variable, and continuous in the time variable and let $u(x,t)$ be a continuous function in $\Omega_T$. Then Eq. \eqref{1-6} 
holds in the viscosity sense if and only if the asymptotic expansion 
\begin{align}\label{1-7}
u(x,t)&=\dfrac{1}{1+M_u(x,t)}\left(\dfrac{\alpha_p}{2}\dashint_{t-\frac{\varepsilon^2}{A_u(x,t)}}^{t}\left\{\max\limits_{y\in\overline{B_\varepsilon(x)}}u(y,s)+\min\limits_{y\in\overline{B_\varepsilon(x)}}u(y,s)\right\}ds\right. \nonumber \\
&\quad+\left.\beta_p\dashint_{t-\frac{\varepsilon^2}{A_u(x,t)}}^t\dashint_{B_\varepsilon(x)}u(y,s)dyds \right) \nonumber\\
&\quad+\dfrac{M_u(x,t)}{1+M_u(x,t)}\left(\dfrac{\alpha_q}{2}\dashint_{t-\frac{\varepsilon^2}{B_u(x,t)}}^{t}\left\{\max\limits_{y\in\overline{B_\varepsilon(x)}}u(y,s)+\min\limits_{y\in\overline{B_\varepsilon(x)}}u(y,s)\right\}ds\right. \nonumber\\
&\quad\left.+\beta_q\dashint_{t-\frac{\varepsilon^2}{B_u(x,t)}}^t\dashint_{B_\varepsilon(x)}u(y,s)dyds \right) \nonumber\\
&\quad+\dfrac{\varepsilon\lvert\nabla u(x,t) \rvert^{q-p}}{4(N+p)(1+M_u(x,t))}\dashint_{B_\varepsilon(x)}u(y+\varepsilon\nabla a(x,t),t)-u(y-\varepsilon\nabla a(x,t),t)dy \nonumber\\
&\quad+o(\varepsilon^2)
\end{align}
as $\varepsilon \rightarrow0$, holds for all $(x,t)\in \Omega_T$ in the viscosity sense. Here
\begin{align}\label{1-8}
\begin{split}
&\alpha_p+\beta_p=1, \quad \frac{\alpha_p}{\beta_p}=\dfrac{p-2}{N+2},\\
&\alpha_q+\beta_q=1, \quad \frac{\alpha_q}{\beta_q}=\dfrac{q-2}{N+2},\\
&M_u(x,t)=a(x,t)\frac{N+q}{N+p}\lvert\nabla u(x,t) \rvert^{q-p},\\
 &\frac{N+p}{A_u(x,t)}+\frac{a(x,t)(N+q)\lvert\nabla u(x,t) \rvert^{q-p}}{B_u(x,t)}=1, \quad A_u(x,t),  B_u(x,t)>0.
\end{split}
\end{align} 
\end{theorem}

\begin{remark}
It is worth mentioning that the positive functions $A(x,t)$ and $B(x,t)$ depend on the test function $\phi$. It means that  $$
\frac{N+p}{A_u(x,t)}+\frac{a(x,t)(N+q)\lvert\nabla u(x,t) \rvert^{q-p}}{B_u(x,t)}=1
$$ 
holds in the viscosity sense, which we called the \textit{viscosity condition}.
\end{remark}

This manuscript is organized as follows. In Section \ref{sec2}, we introduce the basic definitions and give some necessary lemmas that will be used later. In Section \ref{sec3}, we give the proof of Theorem \ref{thm1} and present some corollaries, including the $p(x)$-Laplace type equation. Finally, we prove Theorem \ref{thm2} in Section \ref{sec4}.

\section{Preliminaries}
\noindent
\par
\label{sec2}

In this section, inspired by the ideas developed in \cite{MPR1}, we first give the definition of the asymptotic mean value formula for $u$ at $x\in \Omega$.

\begin{definition}\label{def1}
A continuous function u satisfies 
\begin{align*}
u(x)&=\dfrac{\alpha_p+M_u(x)\alpha_q}{2(1+M_u(x))}\left\{\mathop{\rm{max}}\limits_{\overline{B_\varepsilon(x)}}u+\mathop{\rm{min}}\limits_{\overline{B_\varepsilon(x)}}u\right\}+\dfrac{\beta_p+M_u(x)\beta_q}{1+M_u(x)}\dashint_{B_\varepsilon(x)}u(y)dy\\
&\quad+\dfrac{\varepsilon\lvert\nabla u(x) \rvert^{q-p}}{4(N+p)(1+M_u(x))}\dashint_{B_\varepsilon(x)}u(y+\varepsilon\nabla a(x))-u(y-\varepsilon\nabla a(x))dy+o(\varepsilon^2)
\end{align*}
as $\varepsilon\rightarrow0$, in the viscosiy sense if 
\begin{itemize}
\item [(i)] for every $\phi\in C^2$ such that $u-\phi$ has a strict minimum at the point $x\in\Omega$ with $u(x)=\phi(x)$ and $\nabla\phi(x)\neq0$, we have
\begin{align}\label{2-1}
\begin{split}
0&\geq -\phi(x)+\dfrac{\alpha_p+M_{\phi}(x)\alpha_q}{2(1+M_{\phi}(x))}\left\{\mathop{\rm{max}}\limits_{\overline{B_\varepsilon(x)}}\phi+\mathop{\rm{min}}\limits_{\overline{B_\varepsilon(x)}}\phi\right\}+\dfrac{\beta_p+M_{\phi}(x)\beta_q}{1+M_{\phi}(x)}\dashint_{B_\varepsilon(x)}\phi(y)dy\\
&\quad+\dfrac{\varepsilon\lvert\nabla \phi(x) \rvert^{q-p}}{4(N+p)(1+M_{\phi}(x))}\dashint_{B_\varepsilon(x)}\phi(y+\varepsilon\nabla a(x))-\phi(y-\varepsilon\nabla a(x))dy+o(\varepsilon^2).
\end{split}
\end{align}
\item [(ii)] for every $\phi\in C^2$ such that $u-\phi$ has a strict maximum at the point $x\in\Omega$ with $u(x)=\phi(x)$ and $\nabla\phi(x)\neq0$, we have
\begin{align}\label{2-2}
\begin{split}
0&\leq -\phi(x)+\dfrac{\alpha_p+M_{\phi}(x)\alpha_q}{2(1+M_{\phi}(x))}\left\{\mathop{\rm{max}}\limits_{\overline{B_\varepsilon(x)}}\phi+\mathop{\rm{min}}\limits_{\overline{B_\varepsilon(x)}}\phi\right\}+\dfrac{\beta_p+M_{\phi}(x)\beta_q}{1+M_{\phi}(x)}\dashint_{B_\varepsilon(x)}\phi(y)dy\\
&\quad+\dfrac{\varepsilon\lvert\nabla \phi(x) \rvert^{q-p}}{4(N+p)(1+M_{\phi}(x))}\dashint_{B_\varepsilon(x)}\phi(y+\varepsilon\nabla a(x))-\phi(y-\varepsilon\nabla a(x))dy+o(\varepsilon^2).
\end{split}
\end{align}
\end{itemize}
\end{definition}

Next, we consider the viscosity solution of the double phase elliptic equations. Let us expand the left-hand side of Eq. \eqref{1-1} as follows: 
\begin{align*}
&\quad\text{\rm{div}}(\lvert \nabla u \rvert^{p-2}\nabla u+ a(x)\lvert\nabla u \rvert^{q-2}\nabla u)\\
&=\lvert \nabla u \rvert^{p-2}((p-2)\Delta_{\infty}u+\Delta u)+a(x)\lvert \nabla u \rvert^{q-2}((q-2)\Delta_{\infty}u+\Delta u)\\
&\quad+\lvert \nabla u \rvert^{q-2}\langle \nabla a, \nabla u\rangle,
\end{align*} 
where $\Delta_{\infty}u=|\nabla u|^{-2}\left<D^2u\nabla u,\nabla u\right>$.

Suppose that $u$ is a smooth function with $\nabla u \neq0$, we can see  that $u$ is a solution to Eq.  \eqref{1-1}
if and only if  
\begin{align}\label{2-3}
&-(p-2)\Delta_{\infty}u-\Delta u-a(x)\lvert \nabla u \rvert^{q-p}((q-2)\Delta_{\infty}u+\Delta u)\nonumber \\
&\quad -\lvert \nabla u \rvert^{q-p}\langle \nabla a, \nabla u\rangle=0. 
\end{align} 
 
 \smallskip
 
Then we give the definition of viscosity solutions to Eq.  \eqref{1-1}.

\begin{definition} [\cite{FZ}, Definition 2.5]\label{def2}
Let $1<p\leq q<\infty$ and consider the equation
$$-\text{\rm{div}}(\lvert \nabla u \rvert^{p-2}\nabla u+a(x)\lvert\nabla u \rvert^{q-2}\nabla u)=0.$$

\begin{itemize}
	\item [(i)] A lower semi-continuous function $u$ is a viscosity supersolution if for every $\phi \in C^2$ such that $u-\phi$ has a strict minimum at the point $x\in \Omega$ with $\nabla\phi(x)\neq 0$ we have
\begin{align}
\label{2-4}
&-\left((p-2)\Delta_{\infty}\phi(x)+\Delta \phi(x)\right)-a(x)\lvert\nabla \phi(x)\vert^{q-p}\left((q-2)\Delta_{\infty}\phi(x)+\Delta \phi(x)\right) \nonumber\\
&\quad -\lvert\nabla \phi(x)\vert^{q-p}\langle\nabla a(x),\nabla\phi(x)\rangle \geq0.
\end{align}

\item [(ii)] An upper semi-continuous function $u$ is a viscosity subsolution if for every $\phi \in C^2$ such that $u-\phi$ has a strict maximum at the point $x\in \Omega$ with $\nabla\phi(x)\neq 0$ we have
 \begin{align}\label{2-5}
 \begin{split}
 &-\left((p-2)\Delta_{\infty}\phi(x)+\Delta \phi(x)\right)-a(x)\lvert\nabla \phi(x)\vert^{q-p}\left((q-2)\Delta_{\infty}\phi(x)+\Delta \phi(x)\right)\\
 &\quad -\lvert\nabla \phi(x)\vert^{q-p}\langle\nabla a(x),\nabla\phi(x)\rangle \leq0. 
 \end{split}
 \end{align}

\item [(iii)] Finally, $u$ is a viscosity solution if and only if $u$ is both a viscosity supersolution and a viscosity subsolution.
\end{itemize}
\end{definition}

We next state the following useful results (Lemmas \ref{lem1}--\ref{lem3}), which can be found in \cite[Section 2]{MPR1}.

\begin{lemma}\label{lem1}
Let $\phi$ be a $C^2$ function in a neighborhood of $x$ and let $x_1^{\varepsilon}$ and $x_2^{\varepsilon}$ be the points at which $\phi$ attains its minimum and maximum in $\overline{B_\varepsilon(x)}$ respectively. We have
\begin{align}\label{2-6}
-\phi(x)+\frac{1}{2}\left\{\mathop{\rm{max}}\limits_{\overline{B_\varepsilon(x)}}\phi+\mathop{\rm{min}}\limits_{\overline{B_\varepsilon(x)}}\phi\right\} \geq 
\frac{1}{2}\langle D^2\phi(x)(x^\varepsilon_1-x),(x^\varepsilon_1-x) \rangle+o(\varepsilon^2)
\end{align}
and 
\begin{align}\label{2-7}
-\phi(x)+\frac{1}{2}\left\{\mathop{\rm{max}}\limits_{\overline{B_\varepsilon(x)}}\phi+\mathop{\rm{min}}\limits_{\overline{B_\varepsilon(x)}}\phi\right\}\leq 
\frac{1}{2}\langle D^2\phi(x)(x^\varepsilon_2-x),(x^\varepsilon_2-x) \rangle+o(\varepsilon^2).
\end{align}
\end{lemma}

\begin{lemma}\label{lem2}
Let $\phi$ be a $C^2$ function in a neighborhood of $x$ with $\nabla \phi(x) \neq0$. We have
\begin{align}\label{2-8}
\lim\limits_{\varepsilon\rightarrow0+}\dfrac{x^\varepsilon_1-x}{\varepsilon}=-\dfrac{\nabla\phi}{\lvert\nabla\phi\rvert}(x),
\end{align}
where $x_1^{\varepsilon}$ is defined as in Lemma \ref{lem1}.
\end{lemma}

\begin{lemma}\label{lem3}
Let $\phi$ be a $C^2$ function in a neighborhood of $x$.  We have
\begin{align*}
-\phi(x)+\dashint_{B_\varepsilon(x)}\phi(y)dy=\dfrac{\varepsilon^2}{2(N+2)}\Delta\phi(x)+o(\varepsilon^2) \quad\text{as } \varepsilon\rightarrow0.
\end{align*}
\end{lemma}

Although Lemma \ref{lem1} and Lemma \ref{lem3} provide the bridge between the viscosity solution of $p$-Laplace equation $-\Delta_pu=0$ and the asymptotic mean value formula in \cite{MPR1}, it is not enough for the double phase elliptic equation due to the presence of the term $\left<\nabla a,\nabla u \right>$  in Eq. \eqref{2-3}. Therefore, we need  the following lemma.

\begin{lemma}\label{lem4}
Let $\phi$ be a $C^2$ function in a neighborhood of $x$. We have
\begin{align*}
\dashint_{B_\varepsilon(x)}\phi(y+\varepsilon\nabla a(x))-\phi(y-\varepsilon\nabla a(x))dy=2\varepsilon\langle\nabla\phi(x),\nabla a(x)\rangle+o(\varepsilon^2)
\end{align*}
as $\varepsilon\rightarrow 0$.
\end{lemma}

\begin{proof} Observe that 
\begin{align*}
&\quad \dashint_{B_\varepsilon(x)}\phi(y+\varepsilon\nabla a(x))dy\\
&=\dashint_{B(0,1)}\phi(x+(\nabla a(x)+z)\varepsilon)dz\\
&=\phi(x)+\varepsilon\langle\nabla\phi(x),\nabla a(x) \rangle+\frac{\varepsilon^2}{2}\dashint_{B(0,1)}\langle D^2\phi(x)\nabla a(x),\nabla a(x)\rangle+\langle D^2\phi(x) z,z\rangle dz\\
&\quad+o(\varepsilon^2)
\end{align*}
and 
\begin{align*}
&\quad \dashint_{B_\varepsilon(x)}\phi(y-\varepsilon\nabla a(x))dy\\
&=\dashint_{B(0,1)}\phi(x+(z-\nabla a(x))\varepsilon)dz\\
&=\phi(x)-\varepsilon\langle\nabla\phi(x),\nabla a(x) \rangle+\frac{\varepsilon^2}{2}\dashint_{B(0,1)}\langle D^2\phi(x)\nabla a(x),\nabla a(x)\rangle+\langle D^2\phi(x) z,z\rangle dz\\
&\quad+o(\varepsilon^2).
\end{align*}
Thus, we obtain
$$
\dashint_{B_\varepsilon(x)}\phi(y+\varepsilon\nabla a(x))-\phi(y-\varepsilon\nabla a(x))dy=2\varepsilon\langle\nabla\phi(x),\nabla a(x)\rangle+o(\varepsilon^2) \quad \text{as }\varepsilon\rightarrow 0. 
$$
This finishes the proof.
\end{proof}

\section{Elliptic case}
\noindent
\par
\label{sec3}

In this section, we will prove Theorem \ref{thm1} and consider several special cases as corollaries. Then we apply the ideas to the $p(x)$-Laplace  type equations and give the corresponding conclusions.

\begin{proof} [Proof of Theorem \ref{thm1}]
Considering the sufficiency, we need to show that $u$ is a viscosity solution to Eq. \eqref{1-1} by $u$ satisfying the asymptotic mean value formula. We first prove that $u$ is a viscosity supersolution. To be precise, we intend to prove \eqref{2-4} from \eqref{2-1}.

For the case that $p>2$, we know from \eqref{1-5} that $\alpha_p>0$ and $\alpha_q>0$.
Suppose that the function $u$ satisfies the asymptotic mean value formula in the viscosity sense. Recalling \eqref{2-1}, we have
\begin{align*}
0&\geq -\phi(x)+\dfrac{\alpha_p+M_{\phi}(x)\alpha_q}{2(1+M_{\phi}(x))}\left\{\mathop{\rm{max}}\limits_{\overline{B_\varepsilon(x)}}\phi+\mathop{\rm{min}}\limits_{\overline{B_\varepsilon(x)}}\phi\right\}+\dfrac{\beta_p+M_{\phi}(x)\beta_q}{1+M_{\phi}(x)}\dashint_{B_\varepsilon(x)}\phi(y)dy\\
&\quad+\dfrac{\varepsilon\lvert\nabla \phi(x) \rvert^{q-p}}{4(N+p)(1+M_{\phi}(x))}\dashint_{B_\varepsilon(x)}\phi(y+\varepsilon\nabla a(x))-\phi(y-\varepsilon\nabla a(x))dy+o(\varepsilon^2).
\end{align*}
From $\alpha_p+\beta_p=1, \alpha_q+\beta_q=1$, we write
\begin{align*}
0&\geq\dfrac{\alpha_p+M_{\phi}(x)\alpha_q}{1+M_{\phi}(x)}\Rmnum{1}+\dfrac{\beta_p+M_{\phi}(x)\beta_q}{1+M_{\phi}(x)}\Rmnum{2}+\dfrac{\varepsilon\lvert\nabla \phi(x) \rvert^{q-p}}{4(N+p)(1+M_{\phi}(x))}\Rmnum{3}+o(\varepsilon^2),
\end{align*}
where 
\begin{align*}
&\Rmnum{1}=-\phi(x)+\frac{1}{2}\left\{\mathop{\rm{max}}\limits_{\overline{B_\varepsilon(x)}}\phi+\mathop{\rm{min}}\limits_{\overline{B_\varepsilon(x)}}\phi\right\}, \\
&\Rmnum{2}=-\phi(x)+\dashint_{B_\varepsilon(x)}\phi(y)dy, \\
&\Rmnum{3}=\dashint_{B_\varepsilon(x)}\phi(y+\varepsilon\nabla a(x))-\phi(y-\varepsilon\nabla a(x))dy.
\end{align*}

The non-negativity of $M_{\phi}(x)$ implies that
\begin{align*}
0&\geq \alpha_p\Rmnum{1}+\beta_p\Rmnum{2}+M_{\phi}(x)\left(\alpha_q\Rmnum{1}+\beta_q\Rmnum{2}\right)+\dfrac{\varepsilon\lvert\nabla \phi(x) \rvert^{q-p}}{4(N+p)}\Rmnum{3}+o(\varepsilon^2).
\end{align*}
It follows from Lemmas \ref{lem1}, \ref{lem3} and \ref{lem4} that
\begin{align*}
0&\geq \frac{\alpha_p}{2}\left< D^2\phi(x)(x^\varepsilon_1-x),(x^\varepsilon_1-x) \right>+\dfrac{\varepsilon^2\beta_p}{2(N+2)}\Delta\phi(x) \\
&\quad+M_{\phi}(x)\Bigg(\frac{\alpha_q}{2}\langle D^2\phi(x)(x^\varepsilon_1-x),(x^\varepsilon_1-x) \rangle
+\left.\dfrac{\varepsilon^2\beta_q}{2(N+2)}\Delta\phi(x)\right)\\
&\quad+\dfrac{\varepsilon^2\lvert\nabla \phi(x) \rvert^{q-p}}{2(N+p)}\left<\nabla\phi(x),\nabla a(x)\right>+o(\varepsilon^2).
\end{align*}
Dividing by $\frac{\varepsilon^2}{2}$, taking the limit as $\varepsilon\rightarrow0$ and by Lemma \ref{lem2}, we have
\begin{align*}
0&\geq\alpha_p\Delta_{\infty}\phi(x)+\dfrac{\beta_p}{N+2}\Delta\phi(x)+M_{\phi}(x)\left(\alpha_q\Delta_{\infty}\phi(x)+\dfrac{\beta_q}{N+2}\Delta\phi(x)\right)\\
&\quad+\dfrac{\lvert\nabla \phi(x) \rvert^{q-p}}{N+p}\langle\nabla\phi(x),\nabla a(x)\rangle.
\end{align*}
Multipling by $N+p$, we get
\begin{align*}
0&\geq(p-2)\Delta_{\infty}\phi(x)+\Delta \phi(x)+a(x)\lvert \nabla \phi(x) \rvert^{q-p}((q-2)\Delta_{\infty}\phi(x)+\Delta \phi(x))\\
&\quad+\lvert \nabla \phi(x) \rvert^{q-p}\langle \nabla a(x), \nabla \phi(x)\rangle.
\end{align*}

Therefore, $u$ is a viscosity supersolution according to \eqref{2-4}. We can use \eqref{2-7} instead of \eqref{2-6} to prove that $u$ is a viscosity subsolution and we omit the proof.

For the necessity of the theorem, we need to prove that $u$ satisfies the asymptotic mean value formula in the viscosity sense if $u$ is a viscosity solution to Eq. \eqref{1-1}. Assume that $u$ is a viscosity solution to Eq. \eqref{1-1}.  In particular, $u$ is a viscosity subsolution.  From \eqref{2-5}, we have
\begin{align*}
0&\leq(p-2)\Delta_{\infty}\phi(x)+\Delta \phi(x)+a(x)\lvert \nabla \phi(x) \rvert^{q-p}((q-2)\Delta_{\infty}\phi(x)+\Delta \phi(x))\\
&\quad+\lvert \nabla \phi(x) \rvert^{q-p}\left< \nabla a(x), \nabla \phi(x)\right>.
\end{align*}
By Lemma \ref{lem2},
\begin{align*}
0&\leq(p-2)\left< D^2\phi(x)\left(\dfrac{x^\varepsilon_1-x}{\varepsilon}\right),\left(\dfrac{x^\varepsilon_1-x}{\varepsilon}\right) \right>+\Delta\phi(x)\\
&\quad+a(x)\lvert\nabla \phi(x)\vert^{q-p}\left((q-2)\left< D^2\phi(x)\left(\dfrac{x^\varepsilon_1-x}{\varepsilon}\right)\right.\right.
,\left(\dfrac{x^\varepsilon_1-x}{\varepsilon}\right)\bigg>+\Delta\phi(x)\bigg)\\
&\quad+\lvert\nabla \phi(x)\vert^{q-p}\langle\nabla a(x),\nabla\phi(x)\rangle+o(1).
\end{align*}
Multipling by $\varepsilon^2$ on the inequality above, we get
\begin{align*}
0&\leq(p-2)\langle D^2\phi(x)\left(x^\varepsilon_1-x\right),\left(x^\varepsilon_1-x\right) \rangle+\varepsilon^2\Delta\phi(x)\\
&\quad+a(x)\lvert\nabla \phi(x)\vert^{q-p}\left((q-2)\langle D^2\phi(x)\left(x^\varepsilon_1-x\right),\left(x^\varepsilon_1-x\right)\rangle+\varepsilon^2\Delta\phi(x) \right)\\
&\quad+\varepsilon^2\lvert\nabla \phi(x)\vert^{q-p}\langle\nabla a(x),\nabla\phi(x)\rangle+o(\varepsilon^2).
\end{align*}
By Lemmas \ref{lem1}, \ref{lem3} and  \ref{lem4}, we have 
\begin{align*}
0&\leq 2(p-2)\Rmnum{1}+2(N+2)\Rmnum{2}+a(x)\lvert\nabla \phi(x)\vert^{q-p}(2(q-2)\Rmnum{1}+2(N+2)\Rmnum{2})\\
&\quad+\frac{\varepsilon}{2}\lvert\nabla \phi(x)\vert^{q-p}\Rmnum{3}+o(\varepsilon^2).
\end{align*}
Furthermore, dividing by $2(N+p)$, we obtain
\begin{align*}
0&\leq\left(\alpha_p\Rmnum{1}+\beta_p\Rmnum{2}\right)+a(x)\dfrac{N+q}{N+p}\lvert\nabla \phi(x)\vert^{q-p}\left(\alpha_q\Rmnum{1}+\beta_q\Rmnum{2}\right)+\dfrac{\varepsilon\lvert\nabla \phi(x) \rvert^{q-p}}{4(N+p)}\Rmnum{3}+o(\varepsilon^2).
\end{align*} 
Then separating $\phi(x)$ from $\Rmnum{1}$ and $\Rmnum{2}$, we get
\begin{align*}
(1+M_{\phi}(x))\phi(x)&\leq\dfrac{\alpha_p+M_{\phi}(x)\alpha_q}{2}\left\{\mathop{\rm{max}}\limits_{\overline{B_\varepsilon(x)}}\phi+\mathop{\rm{min}}\limits_{\overline{B_\varepsilon(x)}}\phi\right\}\\
&\quad+\left(\beta_p+M_{\phi}(x)\beta_q\right)\dashint_{B_\varepsilon(x)}\phi(y)dy\\
&\quad+\dfrac{\varepsilon\lvert\nabla \phi(x) \rvert^{q-p}}{4(N+p)}\dashint_{B_\varepsilon(x)}\phi(y+\varepsilon\nabla a(x))-\phi(y-\varepsilon\nabla a(x))dy+o(\varepsilon^2).
\end{align*}
Thus
\begin{align*}
0&\leq -\phi(x)+\dfrac{\alpha_p+M_{\phi}(x)\alpha_q}{2(1+M_{\phi}(x))}\left\{\mathop{\rm{max}}\limits_{\overline{B_\varepsilon(x)}}\phi+\mathop{\rm{min}}\limits_{\overline{B_\varepsilon(x)}}\phi\right\}+\dfrac{\beta_p+M_{\phi}(x)\beta_q}{1+M_{\phi}(x)}\dashint_{B_\varepsilon(x)}\phi(y)dy\\
&\quad+\dfrac{\varepsilon\lvert\nabla \phi(x) \rvert^{q-p}}{4(N+p)(1+M_{\phi}(x))}\dashint_{B_\varepsilon(x)}\phi(y+\varepsilon\nabla a(x))-\phi(y-\varepsilon\nabla a(x))dy+o(\varepsilon^2).
\end{align*} 
Similarly, we can also prove that $u$ satisfies \eqref{2-2} if $u$ is the viscosity supersolution.

When $1<p\leq2$, we can divide it into the following cases: $p=2, q>2$; $p=2, q=2$; $1<p<2, q=2$; $1<p<2, q>2$; $1<q<2$.
The proofs of these cases are similar to the case that $p>2$, by using \eqref{2-7} instead of \eqref{2-6} in Lemma \ref{lem1}  if necessary. 

Combining the arguments above, we complete the proof. 
% (Note that when $p=2$, $\alpha_p=0$ and $\beta_p=1$, the same for $q$).
\end{proof}

From the proof of Theorem \ref{thm1}, the following corollaries will follow.

\begin{corollary}[$p$-Laplace equation]\label{cor1}
Let $1<p<\infty$ and $u(x)$ be a continuous function in a domain $\Omega\subset \mathbb{R}^N$. The equation
$$-\text{\rm{div}}(\lvert \nabla u \rvert^{p-2}\nabla u)=0 \quad \text{in } \Omega$$
holds in the viscosity sense if and only if the asymptotic expansion 
\begin{align*}
u(x)&=\dfrac{\alpha_p}{2}\left\{\mathop{\rm{max}}\limits_{\overline{B_\varepsilon(x)}}u+\mathop{\rm{min}}\limits_{\overline{B_\varepsilon(x)}}u\right\}+\beta_p\dashint_{B_\varepsilon(x)}u(y)dy+o(\varepsilon^2) \quad \text{as }\varepsilon\rightarrow 0
\end{align*}
holds for all $x\in \Omega$ in the viscosity sense. Here $\alpha_p+\beta_p=1,\frac{\alpha_p}{\beta_p}=\frac{p-2}{N+2}.$
\end{corollary}

\begin{remark}
Corollary \ref{cor1} is the main result in \cite{MPR1}. In fact, Corollary \ref{cor1} also holds for $p=\infty$ with $\alpha_p=1$, $\beta_p=0$.
\end{remark}

\begin{corollary}[Variable coefficient $p$-Laplace equation]\label{cor2}
Let $1<p<\infty$, $\tilde a(x)$ be a $C^1$ function in a domain $\Omega \subset \mathbb{R}^N$ with $\tilde a(x)\geq 1$ and let $u(x)$ be a continuous function in $\Omega$. The equation 
\begin{align*}
-\text{\rm{div}}(\tilde a(x)\lvert \nabla u \rvert^{p-2}\nabla u)=0 \quad \text{in } \Omega
\end{align*}
holds in the viscosity sense if and only if the asymptotic expansion 
\begin{align*}
u(x)&=\dfrac{\alpha_p}{2}\left\{\mathop{\rm{max}}\limits_{\overline{B_\varepsilon(x)}}u+\mathop{\rm{min}}\limits_{\overline{B_\varepsilon(x)}}u\right\}+\beta_p\dashint_{B_\varepsilon(x)}u(y)dy\\
&\quad+\dfrac{\varepsilon}{4(N+p)\tilde a(x)}\dashint_{B_\varepsilon(x)}u(y+\varepsilon\nabla \tilde a(x))-u(y-\varepsilon\nabla \tilde  a(x))dy+o(\varepsilon^2)
\end{align*}
as $\varepsilon\rightarrow0$, holds for all $x\in \Omega$ in the viscosity sense. Here $\alpha_p+\beta_p=1,\frac{\alpha_p}{\beta_p}=\frac{p-2}{N+2}$.
\end{corollary}

\begin{proof}  When $q=p$ in Eq.  \eqref{1-1}, we have
$$-\text{\rm{div}}((a(x)+1)\lvert \nabla u \rvert^{p-2}\nabla u)=0.$$
In this situation, we have 
$$
\alpha_p=\alpha_q, \quad \beta_p=\beta_q, \quad M_u(x)=a(x).
$$
Thus, the asymptotic mean value formula \eqref{1-4} reads as 
\begin{align*}
u(x)&=\dfrac{\alpha_p+M_u(x)\alpha_q}{2(1+M_u(x))}\left\{\mathop{\rm{max}}\limits_{\overline{B_\varepsilon(x)}}u+\mathop{\rm{min}}\limits_{\overline{B_\varepsilon(x)}}u\right\}+\dfrac{\beta_p+M_u(x)\beta_q}{1+M_u(x)}\dashint_{B_\varepsilon(x)}u(y)dy\\
&\quad+\dfrac{\varepsilon\lvert\nabla u(x) \rvert^{q-p}}{4(N+p)(1+M_u(x))}\dashint_{B_\varepsilon(x)}u(y+\varepsilon\nabla a(x))-u(y-\varepsilon\nabla a(x))dy+o(\varepsilon^2)\\
&=\dfrac{\alpha_p}{2}\left\{\mathop{\rm{max}}\limits_{\overline{B_\varepsilon(x)}}u+\mathop{\rm{min}}\limits_{\overline{B_\varepsilon(x)}}u\right\}+\beta_p\dashint_{B_\varepsilon(x)}u(y)dy\\
&\quad+\dfrac{\varepsilon}{4(N+p)(1+a(x))}\dashint_{B_\varepsilon(x)}u(y+\varepsilon\nabla a(x))-u(y-\varepsilon\nabla a(x))dy+o(\varepsilon^2).
\end{align*}
Let $\tilde{a}(x)=a(x)+1$. We finish the proof.
\end{proof}

\begin{remark}
In fact, the condition $\tilde a(x)\geq1$ can be replaced by $\tilde a(x)>0$ and we can prove the conclusion by the same method as in Theorem \ref{thm1}.
\end{remark}

Finally, we consider the $p(x)$-Laplace equation
\begin{align}\label{3-1}
-\text{\rm{div}}(\lvert\nabla u\rvert^{p(x)-2}\nabla u)=0 \quad \text{in } \Omega.
\end{align} 
Let us formally expand the left-hand side of Eq. \eqref{3-1} as follows:
\begin{align*}
&\quad \text{\rm{div}}(\lvert\nabla u\rvert^{p(x)-2}\nabla u)\\
&=\sum_{i=1}^{N}\partial_i(\lvert\nabla u\rvert^{p(x)-2}u_i)\\
&=\sum_{i=1}^{N}\partial_i(\lvert\nabla u\rvert^{p(x)-2})u_i+\lvert\nabla u\rvert^{p(x)-2}\Delta u\\
&=\sum_{i=1}^{N}\partial_i(e^{(p(x)-2)\ln\lvert\nabla u\rvert})u_i+\lvert\nabla u\rvert^{p(x)-2}\Delta u \\
&=\sum_{i=1}^{N}\lvert\nabla u\rvert^{p(x)-2}(\ln\lvert\nabla u\rvert\partial_ip+(p(x)-2)\partial_i\ln\lvert\nabla u\rvert)u_i+\lvert\nabla u\rvert^{p(x)-2}\Delta u\\
&=\sum_{i=1}^{N}\lvert\nabla u\rvert^{p(x)-2}(\ln\lvert\nabla u\rvert p_i+(p(x)-2)\lvert\nabla u\rvert^{-2} \sum_{j=1}^{N}u_{ij}u_j)u_i+\lvert\nabla u\rvert^{p(x)-2}\Delta u\\
&=\lvert\nabla u\rvert^{p(x)-2}\left(\ln\lvert\nabla u\rvert\langle\nabla p,\nabla u\rangle+(p(x)-2)\Delta_{\infty}u+\Delta u) \right).
\end{align*}

Suppose that $u$ is a smooth function with $\nabla u \neq0$, we can see that $u$ is a solution to Eq.  \eqref{3-1} if and only if
$$-(p(x)-2)\Delta_{\infty}u-\Delta u-\ln\lvert\nabla u\rvert\langle\nabla p,\nabla u\rangle=0.$$ 

We find that the term $\langle\nabla p,\nabla u\rangle$ appears in the equation above. We can still use Lemma \ref{lem4} to obtain the asymptotic mean value formula by the same method as in Theorem \ref{thm1}, which is given as the following theorem without proof.

\begin{theorem} [$p(x)$-Laplace equation] \label{thm3}
	
Let $p(x)$ be a $C^1$ function in a domain $\Omega \subset \mathbb{R}^N$ with $1<p(x)<\infty $ and $u(x)$ be a continuous function in $\Omega$. Then Eq. \eqref{3-1} holds in the viscosity sense if and only if the asymptotic expansion 
\begin{align*}
u(x)&=\dfrac{\alpha_p(x)}{2}\left\{\mathop{\rm{max}}\limits_{\overline{B_\varepsilon(x)}}u+\mathop{\rm{min}}\limits_{\overline{B_\varepsilon(x)}}u\right\}+\beta_p(x)\dashint_{B_\varepsilon(x)}u(y)dy\\
&\quad+\dfrac{\varepsilon\ln\lvert\nabla u(x)\rvert}{4(N+p(x))}\dashint_{B_\varepsilon(x)}u(y+\varepsilon\nabla p(x))-u(y-\varepsilon\nabla p(x))dy+o(\varepsilon^2)
\end{align*}
as $\varepsilon\rightarrow0$, holds for all $x\in \Omega$ in the viscosity sense. Here $\alpha_p(x)+\beta_p(x)=1, \frac{\alpha_p(x)}{\beta_p(x)}=\frac{p(x)-2}{N+2}$.
\end{theorem}

\section{Parabolic case}
\noindent
\par
\label{sec4}

In this section, we start from the definition of viscosity solutions to the normalized double phase parabolic equation, and combine the ideas in \cite{MPR2} to investigate the possible form of the mean value formula. We integrate the terms with $p$ and $q$ over different time intervals, and find that when these two time lags satisfy the viscosity condition, the mean value formula holds. 

We first give the definition of viscosity solutions to Eq.  \eqref{1-6}. The similar definition can be found in  \cite[Definition 1]{MPR2}.

\begin{definition}\label{def3}
A function $u:\Omega_T\rightarrow\mathbb{R}$ is a viscosity solution to \eqref{1-6} if $u$ is continuous and whenever $(x,t)\in \Omega_T$ and $\phi\in C^2(\Omega_T)$ is such that
\begin{itemize}
	\item [(i)] $u(x,t)=\phi(x,t)$.
    \item [(ii)] $u(y,s)>\phi(y,s)$ for all $(y,s)\in \Omega_T,(y,s)\neq(x,t)$,
then we have at the point $(x,t)$
$$\begin{cases}
		\phi_t \geq(p-2)\Delta_{\infty}\phi+\Delta \phi+a\lvert\nabla \phi\vert^{q-p}\left((q-2)\Delta_{\infty}\phi+\Delta \phi\right)\\ \qquad+\lvert\nabla \phi\vert^{q-p}\langle\nabla a,\nabla\phi\rangle
		&\text{ if } \nabla\phi(x,t)\neq0 ,\\
		\phi_t \geq\lambda_{\min}((p-2)D^2\phi)+\Delta\phi&\text{ if } \nabla\phi(x,t)=0.
\end{cases}$$
\end{itemize}
In addition, when the test function $\phi$  touches $u$ from above, all inequalities are reversed and $\lambda_{\min}((p-2)D^2\phi)$ is replaced by $\lambda_{\max}((p-2)D^2\phi)$.
\end{definition}

In fact, the number of test functions $\phi$ can be reduced, if the gradient of a test function $\phi$ vanishes, we can suppose $D^2\phi=0$. Nothing is required if $\nabla\phi=0$ and $D^2\phi\neq0$. We state the following lemma without proof (see  \cite[Lemma 2]{MPR2} for details).

\begin{lemma}
A function $u:\Omega_T\rightarrow\mathbb{R}$ is a viscosity solution to Eq. \eqref{1-6} if $u$ is continuous and whenever $(x,t)\in \Omega_T$ and $\phi\in C^2(\Omega_T)$ is such that
\begin{itemize}
	\item [(i)] $u(x,t)=\phi(x,t)$.
   \item [(ii)] $u(y,s)>\phi(y,s)$ for all $(y,s)\in \Omega_T,(y,s)\neq(x,t)$,
then at the point $(x,t)$, if $\nabla\phi(x,t)\neq0$, we have 
\begin{align}\label{4-1}
\begin{split}
\phi_t &\geq(p-2)\Delta_{\infty}\phi+\Delta \phi+a\lvert\nabla \phi\vert^{q-p}\left((q-2)\Delta_{\infty}\phi+\Delta \phi\right)\\
&\quad+\lvert\nabla \phi\vert^{q-p}\langle\nabla a,\nabla\phi\rangle;
\end{split}
\end{align}
if $\nabla\phi(x,t)=0$ and $D^2\phi(x,t)=0$, we have 
\begin{align}\label{4-2}
\phi_t(x,t)\geq0.
\end{align}
\end{itemize}
In addition, when the test function $\phi$ touches $u$ from above, all inequalities are reversed.
\end{lemma}

The definition of $u$ satisfying the asymptotic mean value formula \eqref {1-7} at the point $(x,t)$  in the viscosity sense is similar to Definition \ref{def1}, so we omit it. But it is worth to mentioning that $\nabla \phi(x,t)=0$ is allowed in the parabolic case, which is consistent with Definition \ref{def3}.

Similar to the elliptic case, we also need the following lemmas. The ideas of Lemmas \ref{lem5}--\ref{lem7} come from \cite[Section 3]{MPR2} and the proofs are similar.

\begin{lemma}\label{lem5}
Let $\phi$ be a $C^2$ function in a neighborhood of $(x,t)$, $\varepsilon>0,A(x,t)>0,s\in (t-\frac{\varepsilon^2}{A(x,t)},t)$. Denote by $x^{\varepsilon,s}_1,x^{\varepsilon,s}_2$ points in which $\phi$ attains its minimum and maximum over a ball $\overline{B_\varepsilon(x)}$ at time $s$ respectively. We have
\begin{align}\label{4-5}
\begin{split}
&\quad \dfrac{1}{2}\dashint_{t-\frac{\varepsilon^2}{A(x,t)}}^{t}\left\{\mathop{\rm{max}}\limits_{y\in \overline{B_\varepsilon(x)}}\phi(y,s)+\mathop{\rm{min}}\limits_{y\in \overline{B_\varepsilon(x)}}\phi(y,s)\right\}ds-\phi(x,t)\\
&\geq\frac{1}{2}\dashint_{t-\frac{\varepsilon^2}{A(x,t)}}^{t}\left<D^2\phi(x,t)(x^{\varepsilon,s}_1-x),(x^{\varepsilon,s}_1-x)\right>ds-\frac{\varepsilon^2}{2A(x,t)}\phi_t(x,t)+o(\varepsilon^2)
\end{split}
\end{align}
and 
\begin{align}\label{4-6}
\begin{split}
&\quad \dfrac{1}{2}\dashint_{t-\frac{\varepsilon^2}{A(x,t)}}^{t}\left\{\max\limits_{y\in \overline{B_\varepsilon(x)}}\phi(y,s)+\min\limits_{y\in \overline{B_\varepsilon(x)}}\phi(y,s)\right\}ds-\phi(x,t)\\
&\leq\frac{1}{2}\dashint_{t-\frac{\varepsilon^2}{A(x,t)}}^{t}\left<D^2\phi(x,t)(x^{\varepsilon,s}_2-x),(x^{\varepsilon,s}_2-x)\right>ds-\frac{\varepsilon^2}{2A(x,t)}\phi_t(x,t)+o(\varepsilon^2).
\end{split}
\end{align}
\end{lemma}

\begin{lemma}\label{lem6}
Let $\phi$ be a $C^2$ function in a neighborhood of $(x,t)$ with $\nabla\phi(x,t)\neq0$. We have 
\begin{align}\label{4-7}
\lim\limits_{\varepsilon\rightarrow0+}\frac{x^{\varepsilon,s}_1-x}{\varepsilon}=-\dfrac{\nabla\phi}{\lvert\nabla\phi\rvert}(x,t),
\end{align}
where $x^{\varepsilon,s}_1$ is defined as in Lemma \ref{lem5}.
\end{lemma}

\begin{lemma}\label{lem7}
Let $\phi$ be a $C^2$ function in a neighborhood of $(x,t)$, $s$ and $A(x,t)$ are defined as in Lemma \ref{lem5}. Then
\begin{align*}
\dashint_{t-\frac{\varepsilon^2}{A(x,t)}}^t\dashint_{B_\varepsilon(x)}\phi(y,s)dyds-\phi(x,t)=\dfrac{\varepsilon^2}{2(N+2)}\Delta\phi(x,t)-\dfrac{\varepsilon^2}{2A(x,t)}\phi_t(x,t)+o(\varepsilon^2).
\end{align*}
\end{lemma}

\begin{lemma}\label{lem8}
Let $\phi$ be a $C^2$ function in a neighborhood of $(x,t)$. We have 
\begin{align*}
\dashint_{B_\varepsilon(x)}\phi(y+\varepsilon\nabla a(x,t),t)-\phi(y-\varepsilon\nabla a(x,t),t)dy=2\varepsilon\left<\nabla\phi(x,t),\nabla a(x,t)\right>+o(\varepsilon^2).
\end{align*}
\end{lemma}

\medskip

Now we are ready to prove the second main result.

\begin{proof}[Proof of Theorem \ref{thm2}]
%As with the double phase elliptic equation, let us consider the case that $p>2$, where $\alpha_p, \alpha_q >0.$ 

We first prove the sufficiency. If $u$ satisfies the asymptotic mean value formula \ref{1-7} in the viscosity sense, we need to prove that $u$ is a viscosity solution. If so, for a test function $\phi$, we have
\begin{align*}
	0&\geq -\phi(x,t)+\dfrac{1}{1+M_\phi(x,t)}\left(\dfrac{\alpha_p}{2}\dashint_{t-\frac{\varepsilon^2}{A_{\phi}(x,t)}}^{t}\left\{\max\limits_{y\in \overline{B_\varepsilon(x)}}\phi(y,s)+\min\limits_{y\in \overline{B_\varepsilon(x)}}\phi(y,s)\right\}ds\right.\\
	&\quad\left.+\beta_p\dashint_{t-\frac{\varepsilon^2}{A_{\phi}(x,t)}}^t\dashint_{B_\varepsilon(x)}\phi(y,s)dyds \right)\\
	&\quad+\dfrac{M_\phi(x,t)}{1+M_\phi(x,t)}\left(\dfrac{\alpha_q}{2}\dashint_{t-\frac{\varepsilon^2}{B_{\phi}(x,t)}}^{t}\left\{\max\limits_{y\in \overline{B_\varepsilon(x)}}\phi(y,s)+\min\limits_{y\in \overline{B_\varepsilon(x)}}\phi(y,s)\right\}ds\right.\\
	&\quad\left.+\beta_q\dashint_{t-\frac{\varepsilon^2}{B_{\phi}(x,t)}}^t\dashint_{B_\varepsilon(x)}\phi(y,s)dyds \right)\\
	&\quad+\frac{\varepsilon\lvert\nabla \phi(x,t) \rvert^{q-p}}{4(N+p)(1+M_\phi(x,t))}\dashint_{B_\varepsilon(x)}\phi(y+\varepsilon\nabla a(x,t),t)-\phi(y-\varepsilon\nabla a(x,t),t)dy\\
	&\quad+o(\varepsilon^2).
\end{align*}

By the non-negativity of $M_\phi(x,t)$ and splitting $\phi(x,t)$, we get
\begin{align}\label{4-4}
	0&\geq \alpha_p\left(\dfrac{1}{2}\dashint_{t-\frac{\varepsilon^2}{A_{\phi}(x,t)}}^{t}\left\{\max\limits_{y\in \overline{B_\varepsilon(x)}}\phi(y,s)+\min\limits_{y\in \overline{B_\varepsilon(x)}}\phi(y,s)\right\}ds-\phi(x,t)\right) \nonumber \\
	&\quad+\beta_p\left(\dashint_{t-\frac{\varepsilon^2}{A_{\phi}(x,t)}}^t\dashint_{B_\varepsilon(x)}\phi(y,s)dyds-\phi(x,t) \right)\nonumber \\
	&\quad+M_\phi(x,t)\left[\alpha_q\left(\dfrac{1}{2}\dashint_{t-\frac{\varepsilon^2}{B_{\phi}(x,t)}}^{t}\left\{\max\limits_{y\in \overline{B_\varepsilon(x)}}\phi(y,s)+\min\limits_{y\in \overline{B_\varepsilon(x)}}\phi(y,s)\right\}ds-\phi(x,t)\right)\right.\nonumber \\
	&\quad\left.+\beta_q\left(\dashint_{t-\frac{\varepsilon^2}{B_{\phi}(x,t)}}^t\dashint_{B_\varepsilon(x)}\phi(y,s)dyds -\phi(x,t)\right)\right]\nonumber \\
	&\quad+\frac{\varepsilon\lvert\nabla \phi(x,t) \rvert^{q-p}}{4(N+p)}\dashint_{B_\varepsilon(x)}\phi(y+\varepsilon\nabla a(x,t),t)-\phi(y-\varepsilon\nabla a(x,t),t)dy+o(\varepsilon^2),
\end{align}
where $\alpha_p, \alpha_q, \beta_p, \beta_q, M_{\phi}(x,t), A_{\phi}(x,t), B_{\phi}(x,t)$ are determined by \eqref{1-8}.

Assume that $p>2$, where $\alpha_p, \alpha_q>0$. For inequality \eqref{4-4}, we apply Lemmas \ref{lem5}, \ref{lem7} and Lemma \ref{lem8} to have
\begin{align*}
	0&\geq\frac{\alpha_p}{2}\dashint_{t-\frac{\varepsilon^2}{A_{\phi}(x,t)}}^{t}\left<D^2\phi(x,t)(x^{\varepsilon,s}_1-x),(x^{\varepsilon,s}_1-x)\right>ds+\dfrac{\varepsilon^2\beta_p}{2(N+2)}\Delta\phi(x,t)\\
	&\quad-\dfrac{\varepsilon^2}{2A_{\phi}(x,t)}\phi_t(x,t)+M_\phi(x,t)\Bigg(\frac{\alpha_q}{2}\dashint_{t-\frac{\varepsilon^2}{B_{\phi}(x,t)}}^{t}\left<D^2\phi(x,t)(x^{\varepsilon,s}_1-x),(x^{\varepsilon,s}_1-x)\right>ds\\
	&\quad+\dfrac{\varepsilon^2\beta_q}{2(N+2)}\Delta\phi(x,t)-\dfrac{\varepsilon^2}{2B_{\phi}(x,t)}\phi_t(x,t) \Bigg)+\dfrac{\varepsilon^2\lvert\nabla \phi(x,t) \rvert^{q-p}}{2(N+p)}\left<\nabla\phi(x,t),\nabla a(x,t)\right>\\
	&\quad+o(\varepsilon^2).
\end{align*}
When $\nabla\phi(x,t)\neq0$, multiplying by $\dfrac{2}{\varepsilon^2}$ and taking the limit as $\varepsilon\rightarrow 0$ on the inequality above, by Lemma \ref{lem6}, we have 
\begin{align*}
	0&\geq\alpha_p\Delta_{\infty}\phi(x,t)+\dfrac{\beta_p}{N+2}\Delta\phi(x,t)-\frac{1}{A_{\phi}(x,t)}\phi_t(x,t)\\
	&\quad+M_\phi(x,t)\left(\alpha_q\Delta_{\infty}\phi(x,t)+\dfrac{\beta_q}{N+2}\Delta\phi(x,t)-\frac{1}{B_{\phi}(x,t)}\phi_t(x,t)\right)\\
	&\quad+\dfrac{\lvert\nabla \phi(x,t) \rvert^{q-p}}{N+p}\left<\nabla\phi(x,t),\nabla a(x,t)\right>.
\end{align*}
Multipling by $N+p$ again, we get
\begin{align*}
	0&\geq(p-2)\Delta_{\infty}\phi(x,t)+\Delta\phi(x,t)+a(x,t)\lvert\nabla \phi(x,t) \rvert^{q-p}((q-2)\Delta_{\infty}\phi(x,t)\\
	&\quad+\Delta\phi(x,t))-\left(\dfrac{N+p}{A_{\phi}(x,t)}+\dfrac{a(x,t)(N+q)\lvert\nabla \phi(x,t) \rvert^{q-p}}{B_{\phi}(x,t)}\right)\phi_t(x,t)\\
	&\quad+\lvert\nabla \phi(x,t) \rvert^{q-p}\left<\nabla\phi(x,t),\nabla a(x,t)\right>.
\end{align*}
Recalling \eqref{1-8}, we have
\begin{align}\label{4-8}
	\dfrac{N+p}{A_{\phi}(x,t)}+\dfrac{a(x,t)(N+q)\lvert\nabla \phi(x,t) \rvert^{q-p}}{B_{\phi}(x,t)}=1.
\end{align} 

Therefore, we obtain
$$\phi_t \geq(p-2)\Delta_{\infty}\phi+\Delta \phi+a\lvert\nabla \phi\vert^{q-p}\left((q-2)\Delta_{\infty}\phi+\Delta \phi\right)
+\lvert\nabla \phi\vert^{q-p}\langle\nabla a,\nabla\phi\rangle.$$
It follows  that \eqref{4-1} holds when $\nabla\phi(x,t)\neq0$. When $\nabla\phi(x,t)=0$ and $D^2\phi(x,t)=0$, by \eqref{4-8}, we get $A_{\phi}(x,t)=N+p$ and $M_{\phi}(x,t)=0$. According to the asymptotic mean value formula, we have  
\begin{align*}
0&\geq -\phi(x,t)+\frac{\alpha_p}{2}\dashint_{t-\frac{\varepsilon^2}{N+p}}^{t}\left\{\max\limits_{y\in \overline{B_\varepsilon(x)}}\phi(y,s)+\min\limits_{y\in \overline{B_\varepsilon(x)}}\phi(y,s)\right\}ds\\
&\quad+\beta_p\dashint_{t-\frac{\varepsilon^2}{N+p}}^t\dashint_{B_\varepsilon(x)}\phi(y,s) dyds+o(\varepsilon^2).
\end{align*}
By Lemma \ref{lem7} and the expansion 
$$\phi(y,s)-\phi(x,t)=\phi_t(x,t)(s-t)+o(|s-t|+|y-x|^2),$$
we have
\begin{align*}
0&\geq \alpha_p\left(\frac{1}{2}\dashint_{t-\frac{\varepsilon^2}{N+p}}^{t}\left\{\max\limits_{y\in \overline{B_\varepsilon(x)}}\phi(y,s)+\min\limits_{y\in\overline{B_\varepsilon(x)}}\phi(y,s)\right\}ds-\phi(x,t)\right)\\
&\quad-\frac{\varepsilon^2\beta_p}{2(N+p)}\phi_t(x,t)+o(\varepsilon^2)\\
&=\frac{\alpha_p}{2}\dashint_{t-\frac{\varepsilon^2}{N+p}}^t\left\{ \max\limits_{y\in \overline{B_\varepsilon(x)}}(\phi(y,s)-\phi(x,t))+\min\limits_{y\in \overline{B_\varepsilon(x)}}(\phi(y,s)-\phi(x,t))\right\}ds\\
&\quad-\dfrac{\varepsilon^2\beta_p}{2(N+p)}\phi_t(x,t)+o(\varepsilon^2)\\
&=\alpha_p\dashint_{t-\frac{\varepsilon^2}{N+p}}^t\phi_t(x,t)(s-t)ds-\frac{\varepsilon^2\beta_p}{2(N+p)}\phi_t(x,t)+o(\varepsilon^2)\\
&=-\frac{\varepsilon^2\alpha_p}{2(N+p)}\phi_t(x,t)-\frac{\varepsilon^2\beta_p}{2(N+p)}\phi_t(x,t)+o(\varepsilon^2)\\
&=-\frac{\varepsilon^2}{2(N+p)}\phi_t(x,t)+o(\varepsilon^2).
\end{align*}
Dividing by $\varepsilon^2$ and taking the limit as $\varepsilon\rightarrow0$, we have$$\phi_t(x,t)\geq0.$$ Thus, we prove that $u$ is a viscosity supersolution. We can use the same method to prove that $u$ is a viscosity subsolution.
 
For the necessity and other cases, since it is similar to the proof of elliptic case, so we omit it. The proof is complete.
\end{proof}

In particular, we consider the case that $a\equiv 0$. For this case,  it follows from \eqref{1-8}  that $A_u(x,t)=N+p$. Then the following corollary holds.

\begin{corollary} [Normalized parabolic $p$-Laplace equation]\label{cor3}
Let $1<p<\infty$ and $u(x,t)$ be a continuous function in a domain $\Omega_T$. The equation
$$u_t=\lvert\nabla u \rvert ^{2-p}\text{\rm{div}}(\lvert \nabla u \rvert^{p-2}\nabla u) \quad \text{in }\Omega_T$$
holds in the viscosity sense if and only if the asymptotic expansion 
\begin{align*}
u(x,t)&=\dfrac{\alpha_p}{2}\dashint_{t-\frac{\varepsilon^2}{N+p}}^{t}\left\{\mathop{\rm{max}}\limits_{y\in\overline{B_\varepsilon(x)}}u(y,s)+\mathop{\rm{min}}\limits_{y\in\overline{B_\varepsilon(x)}}u(y,s)\right\}ds\\
&\quad+\beta_p\dashint_{t-\frac{\varepsilon^2}{N+p}}^t\dashint_{B_\varepsilon(x)}u(y,s)dyds+o(\varepsilon^2) \quad \text{as } \varepsilon\rightarrow0
\end{align*}
holds for all $(x,t)\in \Omega_T$ in the viscosity sense. Here $\alpha_p+\beta_p=1, \frac{\alpha_p}{\beta_p}=\frac{p-2}{N+2}.$
\end{corollary}

\begin{remark}
Corollary \ref{cor3} is the main result in \cite{MPR2}. In fact, Corollary \ref{cor3} also holds for $p=\infty$  with $\alpha_p=1$, $\beta_p=0$.
\end{remark}

\section*{Acknowledgment}
This work was supported by the National Natural Science Foundation of China (No. 12071098) and the Fundamental Research Funds for the Central Universities (No. 2022FRFK060022).

\end{document}